\documentstyle[amscd,amssymb,verbatim,epsf]{amsart}
\begin{document}

\theoremstyle{plain}
\newtheorem{Thm}{Theorem}
\newtheorem{Cor}{Corollary}
\newtheorem{Con}{Conjecture}
\newtheorem{Main}{Main Theorem}
\newtheorem{Exmp}{Example}
\newtheorem{Lem}{Lemma}
\newtheorem{Prop}{Proposition}
\newtheorem{proof}{Proof}

\theoremstyle{Def}
\newtheorem{Def}{Definition}
\newtheorem{Note}{Note}

\theoremstyle{remark}
\newtheorem{notation}{Notation}
\renewcommand{\thenotation}{}

\errorcontextlines=0
\numberwithin{equation}{section}
\renewcommand{\rm}{\normalshape}%

\title[A new variation on statistical ward  continuity]%
   {A new variation on statistical ward  continuity}
\author[Huseyin Cakalli, Maltepe University, Istanbul-Turkey]{Huseyin Cakalli\\
          Maltepe University, Graduate School of Science and Engineering, Department of Mathematics, Maltepe, Istanbul-Turkey}
\address{Huseyin Cakalli\
          Maltepe University, Graduate School of Science and Engineering, Department of Mathematics, Marmara E\u{g}\.{I}t\.{I}m K\"oy\"u, TR 34857, Maltepe, \.{I}stanbul-Turkey Phone:(+90216)6261050 ext:2311, fax:(+90216)6261113
}
\email{hcakalli@@maltepe.edu.tr; hcakalli@@gmail.com}

\date{\today}

\begin{abstract}
A real valued function defined on a subset $E$ of $\mathbb{R}$, the set of real numbers, is $\rho$-statistically downward   continuous if it preserves $\rho$-statistical downward  quasi-Cauchy sequences of points in $E$, where a sequence $(\alpha_{k})$ of real numbers is called ${\rho}$-statistically downward quasi-Cauchy if $\lim_{n\rightarrow\infty}\frac{1}{\rho_{n} }|\{k\leq n: \Delta \alpha_{k} \geq \varepsilon\}|=0 $
for every $\varepsilon>0$, in which $(\rho_{n})$ is a non-decreasing sequence of positive real numbers tending to $\infty$ such that $\limsup _{n} \frac{\rho_{n}}{n}<\infty $, $\Delta \rho_{n}=O(1)$, and $\Delta \alpha _{k} =\alpha _{k+1} - \alpha _{k}$ for each positive integer $k$. It turns out that a function is uniformly continuous if it is $\rho$-statistical downward   continuous on an above   bounded set.
\end{abstract}

\maketitle

\section{Introduction}
\normalfont

The concept of continuity and any concept involving continuity play a very important role not only in pure mathematics but also in other branches of sciences involving mathematics especially in computer science, information theory, economics, and biological science.

The notion of strongly lacunary convergence or $N_\theta$ convergence was introduced, and studied by Freedman, Sember, and M. Raphael in \cite{FreedmanandSemberandRaphaelSomecesarotypesummabilityspaces} in the sense that a sequence $(\alpha_{k})$ of points in $\mathbb{R}$ is strongly lacunary convergent or $N_\theta$ convergent to an $L\in{\mathbb{R}}$, which is denoted by $N_{\theta}-\lim \alpha_{k}=L$,  if $\lim_{r\rightarrow\infty}\frac{1}{h_{r}}\sum^{}_{k\in{I_{r}}}|\alpha_{k}-L|=0$, where $I_{r}=(k_{r-1},k_{r}]$, and $k_{0}\neq 0$, $h_{r}:=k_{r}-k_{r-1}\rightarrow \infty$ as $r\rightarrow\infty$ and $\theta=(k_{r})$ is an increasing sequence of positive integers (see also \cite{KaplanandCakalliVariationsonstronglylacunaryquasiCauchysequencesAIP}, and \cite{KaplanandCakalliVariationsonstronglacunaryquasiCauchysequencesJNonlinearSciAppl}).
In \cite{FridyandOrhanLacunarystatisconvergence} Fridy and Orhan introduced the concept of lacunary statistically convergence in the sense that a  sequence $(\alpha_{k})$ of points in $\textbf{R}$ is called lacunary statistically convergent, or $S_\theta$-convergent, to an element $L$ of $\textbf{R}$ if
$\lim_{r\rightarrow\infty}\frac{1}{h_{r}}|\{k\in I_{r}: |\alpha_{k}-L|
\geq\varepsilon\}|=0$ for every positive real number $\varepsilon$ where
$I_{r}=(k_{r-1},k_{r}]$ and $k_{0}=0$,
$h_{r}:k_{r}-k_{r-1}\rightarrow \infty$ as $r\rightarrow\infty$
and $\theta=(k_{r})$ is an increasing sequence of positive integers (see also  \cite{FridyandOrhanLacunarystatissummability}, \cite{EsiAsymptoticallydoublelacunryequivalentsequencesdefinedbyOrliczfunctions}, \cite    {EsiandAcikgozMehmetOnalmostlambdastatisticalconvergenceoffuzzynumbers}, \cite{CakalliandKaplanAvariationonlacunarystatisticalquasiCauchysequencesCommunications},  \cite{EsiStatisticalandlacunarystatisticalconvergenceofintervalnumbersintopologicalgroups}, and  \cite{CakalliLacunarystatisticalconvergenceintopgroups}).

Modifying the definitions of a forward Cauchy sequence introduced in \cite{ReillyandSubrahmanyamandVamanamurthyCauchysequencesinquasipseudometricspaces} and  \cite{CollinsandZimmerAnasymmetricArzelaAscolitheorem}, Palladino (\cite{PalladinoOnhalfcauchysequences}) gave the concept of downward half Cauchyness in the following way: a real sequence $(\alpha_{k})$ is called downward half Cauchy if for every $\varepsilon>0$ there exists an $n_{0}\in{\mathbb{N}}$ so that $\alpha_{k}-\alpha_{m} <\varepsilon$ for $m \geq n \geq n_0$ (see also \cite{ReillyandSubrahmanyamandVamanamurthyCauchysequencesinquasipseudometricspaces}). A sequence $(\alpha_{k})$ of points in $\mathbb{R}$, the set of real numbers, is called statistically downward quasi-Cauchy if $\lim_{n\rightarrow\infty}\frac{1}{n}|\{k\leq n: \alpha_{k+1}-\alpha_{k}\geq \varepsilon\}|=0$ for every $\varepsilon>0$ (\cite{CakalliUpwardanddownwardstatisticalcontinuities}).

Recently, many kinds of continuities were introduced and investigated, not all but some of them we recall in the following:
slowly oscillating continuity (\cite{CakalliSlowlyoscillatingcontinuity}, \cite{VallinCreatingslowlyoscillatingsequencesandslowlyoscillatingcontinuousfunctions}), quasi-slowly oscillating continuity (\cite{CanakandDik}), ward continuity (\cite{CakalliForwardcontinuity}, \cite{BurtonandColemanQuasiCauchysequences}), $\delta$-ward continuity (\cite{CakalliDeltaquasiCauchysequences}),  $p$-ward continuity (\cite{CakalliVariationsonquasiCauchysequences}), statistical ward continuity (\cite{CakalliStatisticalwardcontinuity}), $\rho$-statistical ward continuity (\cite{CakalliAvariationonstatisticalwardcontinuity},   lacunary statistical $\delta^{2}$-ward continuity (\cite{YildizIstatistikselboslukludelta2quasiCauchydizileri}), arithmetic continuity (\cite{CakalliAvariationonarithmeticcontinuity}, \cite{CakalliCorrigendumtothepaperentitledAvariationonarithmeticcontinuity})
 Abel continuity (\cite{CakalliandAlbayrakNewtypecontinuitiesviaAbelconvergence}), which enabled some authors to obtain conditions on the domain of a function for some characterizations of uniform continuity in terms of sequences in the sense that a function preserves a certain kind of sequences (see \cite[Theorem 6]{VallinCreatingslowlyoscillatingsequencesandslowlyoscillatingcontinuousfunctions}, \cite[Theorem 1 and Theorem 2]{BurtonandColemanQuasiCauchysequences}, \cite[Theorem 2.3]{CanakandDik}).

The aim of this paper is to introduce, and investigate the concepts of $\rho$-statistical downward continuity and $\rho$-statistical downward compactness.

\section{$\rho$-statistical downward compactness}

The definition of a Cauchy sequence is often misunderstood by the students who first encounter it in an introductory real analysis course. In particular, some fail to understand that it involves far more than that the distance between successive terms is tending to zero. Nevertheless, sequences which satisfy this weaker property are interesting in their own right. In \cite{BurtonandColemanQuasiCauchysequences} the authors call them "quasi-Cauchy", while they  were called "forward convergent to $0$" sequences in \cite{CakalliForwardcontinuity} (see also \cite{CakalliForwardcompactness}), where a sequence $( \alpha_{k} )$ is called quasi-Cauchy if for given any $\varepsilon>0$, there exists an integer $K>0$ such that $k \geq K$ implies that $|\alpha_{k+1}-\alpha_{k} |<\varepsilon$.
A subset $E$ of $\mathbb{R}$ is compact if and only if any sequence in $E$ has a convergent subsequence whose limit is in $E$. Boundedness of a subset $E$ of $\mathbb{R}$ coincides with that any sequence of points in $E$ has either a Cauchy subsequence, or a quasi-Cauchy subsequence.  What is the case for above boundedness?  $\rho$-statistical downward quasi Cauchy sequences provide with the answer.

Weakening the condition on the definition of a $\rho$-statistical quasi-Cauchy sequence, omitting the absolute value symbol, i.e. replacing  $|\alpha_{k+1}-\alpha_{k}|< \varepsilon$ with $\alpha_{k+1}-\alpha_{k} < \varepsilon$ in the definition of a $\rho$-statistical quasi-Cauchy sequence given in \cite{CakalliAvariationonstatisticalwardcontinuity}, we introduce the following definition.

\begin{Def} A sequence $(\alpha_{k})$ of points in $\mathbb{R}$ is called $\rho$-statistically downward quasi-Cauchy if
\[
\lim_{n\rightarrow\infty}\frac{1}{\rho_{n}}|\{k\leq n: \alpha_{k+1}-\alpha_{k}\geq \varepsilon\}|=0
\]
for every $\varepsilon>0$.
\end{Def}
$\Delta \rho ^{-}$ will denote the set of all $\rho$-statistically downward quasi-Cauchy sequences of points in $\mathbb{R}$.
As an example, the sequence $(\rho_{n})$ is a $\rho$-statistically downward quasi-Cauchy sequence.
Any $\rho$-statistically convergent sequence is $\rho$-statistically downward quasi-Cauchy. Any $\rho$-statistically quasi-Cauchy sequence is $\rho$-statistically downward quasi-Cauchy, so any slowly oscillating sequence is $\rho$-statistically downward quasi-Cauchy, so any Cauchy sequence is, so any convergent sequence is. Any downward Cauchy sequence is $\rho$-statistically downward quasi-Cauchy.

Now we give some interesting examples that show importance of the interest.

\begin{Exmp} Let $n$ be a positive integer. In a group of $n$ people, each person selects at random and simultaneously another person of the group. All of the selected persons are then removed from the group, leaving a random number $n_{1} < n$ of people which form a new group. The new group then repeats independently the selection and removal thus described, leaving $n_{2} < n_{1}$ persons, and so forth until either one person remains, or no persons remain. Denote by $\alpha_n$ the probability that, at the end of this iteration initiated with a group of $n$ persons, one person remains. Then the sequence $(\alpha_{1}, \alpha_{2}, · · ·, \alpha_{n},...)$  is a $\rho$-statistically downward quasi-Cauchy sequence (see also \cite{WinklerMathematicalPuzzles}).
\end{Exmp}

\begin{Exmp} In a group of $k$ people, $k$ is a positive integer, each person selects independently and at random one of three subgroups to which to belong, resulting in three groups with random numbers $k_{1}$, $k_{2}$, $k_{3}$ of members; $k_{1} + k_{2} + k_{3} = k$. Each of the subgroups is then partitioned independently in the same manner to form three sub subgroups, and so forth. Subgroups having no members or having only one member are removed from the process. Denote by $\alpha_{k}$ the expected value of the number of iterations up to complete removal, starting initially with a group of $k$ people. Then the sequence $(\alpha_{1}, \frac{\alpha_{2}}{2}, \frac{\alpha_{3}}{3},...,\frac{\alpha_{n}}{n},...)$ is a bounded non-convergent $\rho$-statistically downward quasi-Cauchy sequence (\cite{KeaneUnderstandingErgodicity}).
\end{Exmp}

It is well known that a subset of $\mathbb{R}$ is compact if and only if any sequence of points in $E$ has a convergence subsequence, whose limit is in $E$. By using this idea in the definition of sequential compactness, now we introduce a definition of $\rho$-statistically downward compactness of a subset of $\mathbb{R}$.

\begin{Def}
A subset $E$ of $\mathbb{R}$ is called $\rho$-statistically downward compact if any sequence of points in $E$ has a $\rho$-statistically downward quasi-Cauchy subsequence.
\end{Def}

First, we note that any finite subset of $\mathbb{R}$ is $\rho$-statistically downward compact, the union of finite number of $\rho$-statistically downward compact subsets of $\mathbb{R}$ is $\rho$-statistically downward compact, and the intersection of any family of $\rho$-statistically downward compact subsets of $\mathbb{R}$ is $\rho$-statistically downward compact. Furthermore any subset of a $\rho$-statistically downward compact set is $\rho$-statistically downward compact, any compact subset of $\mathbb{R}$ is $\rho$-statistically downward compact, any bounded subset of $\mathbb{R}$ is $\rho$-statistically downward compact, and any slowly oscillating compact subset of $\mathbb{R}$ is $\rho$-statistically downward compact (see \cite{CakalliSlowlyoscillatingcontinuity} for the definition of slowly oscillating compactness). The sum of finite number of $\rho$-statistically downward compact subsets of $\mathbb{R}$ is $\rho$-statistically downward compact. Any bounded above subset of $\mathbb{R}$ is $\rho$-statistically downward compact. These observations suggest to us giving the following result.

\begin{Thm} \label{TheoremLambdadownwardtatisticalwardcompactiffboundedabove}
A subset of $\mathbb{R}$ is $\rho$-statistically downward compact if and only if it is bounded above.
\end{Thm}

\begin{proof}
Let $E$ be a bounded above subset of $\mathbb{R}$. If $E$ is also bounded below, then it follows from \cite[Lemma 2]{CakalliStatisticalwardcontinuity} and \cite[Theorem 3]{CakalliStatisticalquasiCauchysequences} that any sequence of points in $E$ has a quasi Cauchy subsequence which is also $\rho$-statistically downward quasi-Cauchy. If $E$ is unbounded below, and $(\alpha_{n})$ is an unbounded below sequence of points in $E$, then for $k=1$ we can find an $\alpha_{n_{1}}$ less than $0$. For k=2 we can pick an $\alpha_{n_{2}}$ such that $\alpha_{n_{2}}<-\rho_{2}+\alpha_{n_{1}}$. We can successively find  for each $k\in{\mathbb{N}}$ an $\alpha_{n_{k}}$ such that $\alpha_{n_{k+1}}<-\rho_{k+1}+\alpha_{n_{k}}$.  Then $\alpha_{n_{k+1}}-\alpha_{n_{k}}<-\rho_{k+1}$ for each $k \in {\mathbb{N}}$. Therefore we see that
\[
\lim_{n\rightarrow\infty}\frac{1}{\rho_{n}}|\{k\leq n: \alpha_{n_{k+1}}-\alpha_{n_{k}}\geq \varepsilon\}|=0
\]
for every $\varepsilon>0$. Conversely, suppose that $E$ is not bounded above. Pick an element $\alpha_{1}$ of $E$. Then we can choose an element $\alpha_{2}$ of $E$ such that $\alpha_{2}>\rho_{2}+\alpha_{1}$. Similarly we can choose an element $\alpha_{3}$ of $E$ such that $\alpha_{3}>\rho_{3}+\alpha_{2}$. We can inductively choose  $\alpha_{k}$ satisfying
$\alpha_{k+1}>\rho_{k}+\alpha_{k}$ for each positive integer $k$. Then the sequence $(\alpha_{k})$ does not have any $\rho$-statistically downward quasi-Cauchy subsequence. Thus $E$ is not $\rho$-statistically downward  compact. This contradiction completes the proof.
\end{proof}

It follows from the above theorem that if a closed subset $E$ of $\mathbb{R}$ is $\rho$-statistically downward  compact, and $-A$ is $\rho$-statistically downward  compact, then any sequence of points in $E$ has a $(P_{n} ,s)$-absolutely almost convergent subsequence (see \cite{CakalliandTaylanOnabsolutelyalmostconvergence}, \cite{OzarslanandYildizAnewstudyontheabsolutesummabilityfactorsofFourierseries}, \cite{YildizAnewtheoremonlocalpropertiesoffactoredFourierseries}, \cite{BorOnGeneralizedAbsoluteCesaroSummability}, \cite{YildizOnApplicationofMatrixSummabilitytoFourierSeriesMathMethodsApplSci}, \cite{YildizOnAbsoluteMatrixSummabilityFactorsofInfiniteSeriesandFourierSeries}, and \cite{YildizIstatistikselboslukludelta2quasiCauchydizileri}).

\begin{Cor} A subset of $\mathbb{R}$ is $\rho$-statistically downward compact if and only if it is statistically downward compact.
\end{Cor}
\begin{proof} The proof of this theorem follows from Theorem \ref{TheoremLambdadownwardtatisticalwardcompactiffboundedabove} and \cite[Theorem 3.3]{CakalliUpwardanddownwardstatisticalcontinuities}, so it is omitted.
\end{proof}

\begin{Cor} A subset of $\mathbb{R}$ is $\rho$-statistically downward compact if and only if it is lacunary statistically downward compact.
\end{Cor}
\begin{proof} The proof of this theorem follows from Theorem \ref{TheoremLambdadownwardtatisticalwardcompactiffboundedabove} and \cite[Theorem 1.9]{CakalliandMucukLacunarystatisticallydownwardanddownwardhalfquasiCauchysequences}, so is omitted.

\end{proof}

\begin{Cor} A subset of $\mathbb{R}$ is $\rho$-statistically downward compact if and only if it is downward compact.
\end{Cor}
\begin{proof} The proof of this theorem follows from Theorem \ref{TheoremLambdadownwardtatisticalwardcompactiffboundedabove} and \cite[Theorem 2.8 ]{CakalliBeyondCauchyandquasiCauchysequences}, so is omitted.

\end{proof}

\begin{Cor} A subset $A$ of $\mathbb{R}$ is bounded if and only if the sets $A$ and $-A$ are $\rho$-statistically downward compact.
\end{Cor}
\begin{proof} The proof of this theorem follows from the fact that a subset $A$ of $\mathbb{R}$ is bounded if and only if the sets $A$ and $-A$ are bounded above, so is omitted.
\end{proof}

\section{$\rho$-statistical downward continuity}

A real valued function $f$ defined on a subset of $\mathbb{R}$ is $\rho$-statistically continuous, or $S_{\rho}$-continuous if for each point $\ell$ in the domain, $S_{\rho}-\lim_{n\rightarrow\infty}f(\alpha_{k})=f(\ell)$ whenever $S_{\rho}-\lim_{n\rightarrow\infty}\alpha_{k}=\ell$ (\cite{CakalliAvariationonstatisticalwardcontinuity}). This is equivalent to the statement that $(f(\alpha_{k}))$ is a convergent sequence whenever $(\alpha_{k})$ is. This is also equivalent to the statement that $(f(\alpha_{k}))$ is a Cauchy sequence whenever $(\alpha_{k})$ is Cauchy provided that the domain of the function is complete. These known results for $\rho$-statistically-continuity and continuity for real functions in terms of sequences might suggest to us introducing a new type of continuity, namely, $\rho$-statistically-downward continuity, weakening the condition on the definition of a $\rho$-statistically ward continuity, omitting the absolute value symbol, i.e. replacing "$|\alpha_{k+1}- \alpha_{k}|$" with "$\alpha_{k+1}- \alpha_{k}$" in the definition of $\rho$-statistically ward continuity given in \cite{CakalliAvariationonstatisticalwardcontinuity}.
\begin{Def}
A real valued function $f$ is called \textit{$\rho$-statistically downward continuous}, or \textit{$S^{-}_{\rho}$-continuous} on a subset $E$ of $\mathbb{R}$ if it preserves $\rho$-statistically downward quasi-Cauchy sequences, i.e.  the sequence $(f(\alpha_{k}))$ is $\rho$-statistically-downward quasi-Cauchy whenever $(\alpha_{k})$ is a $\rho$-statistically-downward quasi-Cauchy sequence of points in $E$, writing in symbols,

$\lim_{n\rightarrow\infty}\frac{1}{\rho_{n} }|\{k\leq n: f(\alpha_{k+1})-f(\alpha_{k}) \geq\varepsilon\}|=0$, \\whenever $(\alpha_{k})$ is a sequence of points in $E$ such that

$\lim_{n\rightarrow\infty}\frac{1}{\rho_{n} }|\{k\leq n: \alpha_{k+1}-\alpha_{k} \geq\varepsilon\}|=0$ \\for every positive real number $\varepsilon$.
\end{Def}

It should be noted that $\rho$-statistically-downward continuity cannot be given by any $A$-continuity in the manner of \cite{ConnorandGrosseErdmannSequentialdefinitionsofcontinuityforrealfunctions}. We see that the composition of two $\rho$-statistically-downward continuous functions is $\rho$-statistically-downward continuous, and for every positive real number $c$, $cf$ is $\rho$-statistically-downward continuous, if $f$ is $\rho$-statistically-downward continuous.\\
We see in the following that the sum of two $\rho$-statistically-downward continuous functions is $\rho$-statistically-downward  continuous
\begin{Prop} \label{PropThesumoflambdadownwardcontinuousfunctionsislambdadownwardcontinuous}
If $f$ and $g$ are $\rho$-statistically-downward continuous functions, then $f+g$ is $\rho$-statistically-downward continuous.
\end{Prop}
\begin{proof}
Let $f$, $g$ be $\rho$-statistically-downward continuous functions on a subset $E$ of $\mathbb{R}$. To prove that $f+g$ is $\rho$-statistically-downward continuous on $E$, take any $\rho$-statistically-downward quasi-Cauchy sequence $(\alpha_{k})$ of points in $E$. Then $(f(\alpha_{k}))$ and $(g(\alpha_{k}))$ are $\rho$-statistically-downward quasi-Cauchy sequences. Let $\varepsilon>0$ be given. Since $(f(\alpha_{k}))$ and $(g(\alpha_{k}))$ are $\rho$-statistically-downward quasi-Cauchy, we have
\[
\lim_{n\rightarrow\infty}\frac{1}{\rho_{n} }|\{k\leq n: f(\alpha_{k+1})-f(\alpha_{k}) \geq\frac{\varepsilon}{2}\}|=0
\]
and
\[
\lim_{n\rightarrow\infty}\frac{1}{\rho_{n} }|\{k\leq n: g(\alpha_{k+1}) - g(\alpha_{k}) \geq\frac{\varepsilon}{2}\}|=0
.\]
Hence \[
\lim_{n\rightarrow\infty}\frac{1}{\rho_{n} }|\{k\leq n: [f(\alpha_{k+1}) - f(\alpha_{k})]+ [g(\alpha_{k+1} - g(\alpha_{k})]) \geq \varepsilon\}|=0
\]
which follows from the inclusion\\ $\{k\leq n: (f+g)(\alpha_{k+1}) - (f+g)(\alpha_{k} \geq \varepsilon\}\subseteq{\{k\leq n:  f(\alpha_{k+1}) - f(\alpha_{k}) \geq\frac{\varepsilon}{2}\} \cup {\{k\leq n: g(\alpha_{k+1}) - g(\alpha_{k}) \geq\frac{\varepsilon}{2}\}}}$. \\
This completes the proof.
\end{proof}

\begin{Prop} \label{PropThecompositeoftwooflambdadownwardcontinuousfunctionsislambdadownwardcontinuous}
The composition of two $\rho$-statistically-downward continuous functions is $\rho$-statistically-downward continuous, i.e. if $f$ and $g$ are $\rho$-statistically-downward continuous functions on $\mathbb{R}$, then the composition $gof$ of $f$ and $g$ is $\rho$-statistically-downward continuous.
\end{Prop}

\begin{proof}
Let $f$ and $g$ be $\rho$-statistically-downward continuous functions on $\mathbb{R}$, and $(\alpha _{n})$ be a $\rho$-statistically-downward quasi Cauchy sequence of points in $\mathbb{R}$. As $f$ is $\rho$-statistically-downward continuous, the transformed sequence $(f(\alpha _{n}))$ is a $\rho$-statistically downward quasi Cauchy sequence. Since $g$ is $\rho$-statistically-downward continuous, the transformed sequence $g(f(\alpha _{n}))$ of the sequence $(f(\alpha _{n}))$ is a $\rho$-statistically downward quasi Cauchy sequence. This completes the proof of the theorem.
\end{proof}

In connection with $\rho$-statistically-downward quasi-Cauchy sequences, $\rho$-statistically-quasi-Cauchy sequences, $\rho$-statistically-statistical convergent sequences, and convergent sequences the problem arises to investigate the following types of  "continuity" of functions on $\mathbb{R}$:
\begin{description}
\item[($\delta S^{-} _{\rho} $)] $(\alpha_{k}) \in {\Delta S^{-} _{\rho}} \Rightarrow  (f(\alpha_{k})) \in {\Delta S^{-} _{\rho}}$
\item[($\delta S^{-} _{\rho} c$)] $(\alpha_{k}) \in {\Delta S^{-} _{\rho}} \Rightarrow  (f(\alpha_{k})) \in {c}$
\item[$(c)$] $(\alpha_{k}) \in {c} \Rightarrow  (f(\alpha_{k})) \in {c}$
\item[$(c\delta S^{-} _{\rho})$] $(\alpha_{k}) \in {c} \Rightarrow  (f(\alpha_{k})) \in {\Delta S^{-} _{\rho}}$
\item[$(S_{\rho})$] $(\alpha_{k}) \in {S_{\rho}} \Rightarrow  (f(\alpha_{k})) \in {S_{\rho}}$
\item[($\delta S_{\rho})$] $(\alpha_{k}) \in {\Delta S _{\rho}} \Rightarrow  (f(\alpha_{k})) \in {\Delta S_{\rho}}$.
\end{description}
We see that $(\delta S^{-} _{\rho})$ is $\rho$-statistically downward continuity of $f$, $(S_{\rho})$ is the $\rho$-statistical continuity, and $(\delta S_{\rho})$ is the $\rho$-statistically-ward continuity. It is easy to see that $(\delta S^{-} _{\rho} c)$ implies $(\delta S^{-} _{\rho})$; $(\delta S^{-} _{\rho})$ does not imply $(\delta S^{-} _{\rho} c)$; $(\delta S^{-} _{\rho})$ implies $(c\delta S^{-} _{\rho})$; $(c\delta S^{-} _{\rho})$ does not imply $(\delta S^{-} _{\rho})$; $(\delta S^{-} _{\rho} c)$ implies $(c)$, and $(c)$ does not imply $(\delta S^{-} _{\rho} c)$; and $(c)$ implies $(c\delta S^{-}_{\rho})$. We see that $(c)$ can be replaced by not only $\rho$-statistically-continuity (\cite{CakalliAvariationonstatisticalwardcontinuity}), but also statistically-continuity (\cite{CakalliStatisticalquasiCauchysequences}),
lacunary statistical continuity (\cite{CakalliandArasandSonmezLacunarystatisticalwardcontinuity}), $N_{\theta}$-sequential continuity (\cite{CakalliandKaplanAstudyonNthetaquasiCauchysequences}), $I$-sequential continuity (\cite{CakalliandHazarikaIdealquasiCauchysequences}), and more generally  $G$-sequential continuity (\cite{CakalliSequentialdefinitionsofcompactness}, \cite{CakalliOnGcontinuity}).\\
Now we give the implication $(\delta S^{-} _{\rho})$ implies ($\delta S_{\rho}$), i.e. any $\rho$-statistically-downward continuous function is $\rho$-statistically-ward continuous.
\begin{Thm} \label{TheoremLambdatatisticaldownwardcontinuityimplieslambdastatisticalwardcontinuity}  If $f$ is $\rho$-statistically-downward continuous on a subset $E$ of $\mathbb{R}$, then it is $\rho$-statistically-ward continuous on $E$.
\end{Thm}
\begin{proof}
Let $(\alpha_{k})$ be any $\rho$-statistically-quasi-Cauchy sequence of points in $E$. Then the sequence $$(\alpha_{1},\alpha_{2},\alpha_{1},\alpha_{2},\alpha_{3},\alpha_{2}, \alpha_{3},...,\alpha_{n-1},\alpha_{k},\alpha_{n-1},\alpha_{k},\alpha_{n+1},\alpha_{k},\alpha_{n+1},...)$$ is also $\rho$-statistically quasi-Cauchy. Then it is $\rho$-statistically-downward quasi-Cauchy. As $f$ is $\rho$-statistically-downward continuous, the sequence $$(f(\alpha_{1}),f(\alpha_{2}),f(\alpha_{1}),f(\alpha_{2}),f(\alpha_{3}),f(\alpha_{2}),f(\alpha_{3}),...,f(\alpha_{n-1}),$$ $$f(\alpha_{k}),f(\alpha_{n-1}),f(\alpha_{k}),f(\alpha_{n+1}),f(\alpha_{k}),f(\alpha_{n+1}),...)$$ is $\rho$-statistically-downward quasi-Cauchy. It follows from this that
\[ \lim_{n\rightarrow\infty}\frac{1}{\rho_{n} }|\{k\leq n: |f(\alpha_{k+1})-f(\alpha_{k})| \geq \varepsilon \}|=0 \] for every $\varepsilon>0$. This completes the proof of the theorem.
\end{proof}
We note that the converse of the preceding theorem is not always true, i.e. there are $\rho$-statistically-ward continuous functions which are not $\rho$-statistically-downward continuous.
\begin{Exmp}
Write $\boldsymbol{\rho}=(\rho_{n})$ as $\rho_{n}=n+\frac{1}{n}$ for each $n>1$, $\rho_{1}=1$, and consider the sequence $\boldsymbol{\alpha}=(n)$. Then we see that the function $f$ defined by $f(x)=-x$ for every $x \in {\mathbb{R}}$ is an $\rho$-statistically-ward continuous function, but not $\rho$-statistically-downward continuous.
\end{Exmp}
Now we give the implication $(\delta S^{-} _{\rho})$ implies $(S _{\rho})$, i.e. any $\rho$-statistically-downward continuous function is $\rho$-statistically-continuous.
\begin{Cor} \label{TheoremLambdastatisticallydownwardcontinuityimplieslambdastatisticalcontinuity}  If $f$ is $\rho$-statistically-downward continuous on a subset $E$ of $\mathbb{R}$, then it is $\rho$-statistically-continuous on $E$.
\end{Cor}
\begin{proof}
Although the proof follows from \cite[Theorem 3]{CakalliAvariationonstatisticalwardcontinuity}, the preceding theorem, and the fact that $\rho$-statistical continuity coincides with ordinary continuity, we give a direct proof in the following for completeness. Let $(\alpha_{k})$ be any $\rho$-statistically convergent sequence with $S_{\rho}-\lim_{k\rightarrow\infty}\alpha_{k}=\ell$. Then $$(\alpha_{1}, \ell , \alpha_{1}, \ell , \alpha_{2}, \ell, \alpha_{2}, \ell,..., \alpha_{k}, \ell, \alpha_{k}, \ell,...)$$ is also $\rho$-statistically convergent to $\ell$. Thus it is $\rho$-statistically-downward quasi-Cauchy. Hence $$(f(\alpha_{1}), f(\ell), f(\alpha_{1}),f(\ell), f(\alpha_{2}), f(\ell), f(\alpha_{2}), f(\ell),...,f(\alpha_{k}), f(\ell), f(\alpha_{k}), f(\ell),...)$$ is $\rho$-statistically-downward quasi-Cauchy. It follows from this that
\[ \lim_{n\rightarrow\infty}\frac{1}{\rho_{n} }|\{k\leq n: |f(\alpha_{k})-f(\ell)| \geq \varepsilon \}|=0 \] for every $\varepsilon>0$. This completes the proof.
\end{proof}
Observing that $\rho$-statistically-continuity implies ordinary continuity, we note that it follows from Theorem \ref{TheoremLambdastatisticallydownwardcontinuityimplieslambdastatisticalcontinuity} that $\rho$-statistically-downward continuity implies not only ordinary continuity, but also some other kinds of continuities, namely, lacunary statistical continuity, statistical continuity (\cite{CakalliAvariationonstatisticalwardcontinuity}), $N_{\theta}$-sequential continuity, $I$-continuity for any non-trivial admissible ideal $I$ of $\mathbb{N}$ (\cite[Theorem 4]{CakalliandHazarikaIdealquasiCauchysequences}, \cite{CakalliAvariationonwardcontinuity}), and $G$-continuity for any regular subsequential method $G$ (see \cite{ConnorandGrosseErdmannSequentialdefinitionsofcontinuityforrealfunctions}, \cite{CakalliSequentialdefinitionsofcompactness}, and \cite{CakalliOnGcontinuity}).
\begin{Thm} \label{TheoremLambdadownwardtatisticalcontinousimageoflambdadownwardstatisticalwardcompactsubset} $\rho$-statistically-downward continuous image of any $\rho$-statistically-downward compact subset of $\mathbb{R}$ is $\rho$-statistically-downward  compact.
\end{Thm}
\begin{proof}
Let $E$ be a subset of $\mathbb{R}$, $f:E\longrightarrow$ $\mathbb{R}$ be an $\rho$-statistically-downward continuous function, and $A$ be an $\rho$-statistically-downward compact subset of $E$. Take any sequence $\boldsymbol{\beta}=(\beta_{n})$ of points in $f(A)$. Write $\beta_{n}=f(\alpha_{k})$, where $\alpha_{k}\in {A}$ for each $n \in{\mathbb{N}}$, $\boldsymbol{\alpha}=(\alpha_{k})$. $\rho$-statistically-downward compactness of $A$ implies that there is an $\rho$-statistically-downward quasi-Cauchy subsequence $\boldsymbol{\xi}$ of the sequence of $\boldsymbol{\alpha}$. Write $\boldsymbol{\eta}=(\eta_{k})=f(\boldsymbol{\xi})=(f(\xi_{k}))$. Then $\boldsymbol{\eta}$ is an $\rho$-statistically-downward quasi-Cauchy subsequence of the sequence $\boldsymbol{\beta}$. This completes the proof of the theorem.
\end{proof}
\begin{Thm} \label{TheoremLambdadownwardstatisticalcontinuityonlambdastatisticallydownwardcompactsubsetimpliesuniformcontinuity} Any $\rho$-statistically-downward continuous real valued function on a $\rho$-statistically-downward compact subset of $\mathbb{R}$ is uniformly continuous.
\end{Thm}
\begin{proof}
Let $E$ be an  $\rho$-statistically-downward compact subset of $\mathbb{R}$ and let $f:E\longrightarrow$ $\mathbb{R}$. Suppose that $f$ is not uniformly continuous on $E$ so that there exists an  $\varepsilon_{0} > 0$ such that for any $\delta >0$, there are $x, y \in{E}$ with $|x-y|<\delta$ but $|f(x)-f(y)| \geq \varepsilon_{0}$. For each positive integer $n$, there are $\alpha_{n}$ and $\beta_{n}$ such that $|\alpha_{n}-\beta_{n}|<\frac{1}{n}$, and $|f(\alpha_{n})-f(\beta_{n})|\geq \varepsilon_{0}$. Since $E$ is $\rho$-statistically-downward compact, there exists an $\rho$-statistically-downward quasi-Cauchy subsequence $(\alpha_{n_{k}})$ of the sequence $(\alpha_{k})$. It is clear that the corresponding subsequence $(\beta_{n_{k}})$ of the sequence $(\beta_{n})$ is also $\rho$-statistically-downward quasi-Cauchy, since $(\beta_{n_{k+1}}-\beta_{n_{k}})$ is a sum of three $\rho$-statistically-downward quasi-Cauchy sequences, i.e. $$\beta_{n_{k+1}}-\beta_{n_{k}}=(\beta_{n_{k+1}}-\alpha_{n_{k+1}})+(\alpha_{n_{k+1}}-\alpha_{n_{k}})+(\alpha_{n_{k}}-\beta_{n_{k}}).$$ Then the sequence $$(\beta_{n_{1}}, \alpha_{n_{1}},  \beta_{n_{2}}, \alpha_{n_{2}}, \beta_{n_{3}}, \alpha_{n_{3}}, ..., \beta_{n_{k}}, \alpha_{n_{k}}, ...)$$ is $\rho$-statistically-downward quasi-Cauchy, since the sequence $(\beta_{n_{k+1}}-\alpha_{n_{k}})$ is a $\rho$-statistically-downward quasi-Cauchy sequence which follows from the equality $$\beta_{n_{k+1}}-\alpha_{n_{k}}=\beta_{n_{k+1}}-\beta_{n_{k}}+\beta_{n_{k}}-\alpha_{n_{k}}.$$ But the sequence $$(f(\beta_{n_{1}}), f(\alpha_{n_{1}}), f(\beta_{n_{2}}), f(\alpha_{n_{2}}), f(\beta_{n_{3}}), f(\alpha_{n_{3}}),..., f(\beta_{n_{k}}), f(\alpha_{n_{k}}),...)$$ is not $\rho$-statistically-downward quasi-Cauchy. Thus $f$ does not preserve $\rho$-statistically-downward quasi-Cauchy sequences. This contradiction completes the proof of the theorem.
\end{proof}

\begin{Thm} \label{TheoremuniformlycontinuousfunctiononEtransformsquasiCauchytolamdquasiCauchy}
If a function $f$ is uniformly continuous on a subset $E$ of $\mathbb{R}$, then $(f(\alpha_{k}))$ is $\rho$-statistically-downward quasi Cauchy whenever $(\alpha_{k})$ is a quasi-Cauchy sequence of points in $E$.
\end{Thm}
\begin{proof} Let $E$ be a subset of $\mathbb{R}$ and let $f$ be a uniformly continuous function on $E$. Take any quasi-Cauchy sequence $(\alpha_{k})$ of points in $A$, and let $\varepsilon$ be any positive real number in $]0,1[$. By uniform continuity  of $f$, there exists a $\delta>0$ such that $|f(x)-f(y)|<\varepsilon$ whenever $|x-y|<\delta$ and $x, y \in{E}$. Since $(\alpha_{k})$ is a quasi-Cauchy sequence, there exists a positive integer $k_{0}$ such that $|\alpha_{k+1}-\alpha_{k}|<\varepsilon$ for $k\geq k_{0}$, therefore $f(\alpha_{k+1})-f(\alpha_{k})<\varepsilon$ for $k\geq k_{0}$.
Thus the number of indices $k\leq n$ that satisfy $f(\alpha_{k+1})-f(\alpha_{k}) \geq \varepsilon$ is less than or equal to $k_{0}$.
Hence $$\frac{1}{\rho_{n}} |\{k\leq n: f(\alpha_{k+1})-f(\alpha_{k}) \geq \varepsilon\}| \leq \frac{k_{0}}{\rho_{n}}.$$
It follows from this that
\[ \lim_{n\rightarrow\infty}\frac{1}{\rho_{n} }|\{k\leq n: f(\alpha_{k+1})-f(\alpha_{k}) \geq \varepsilon \}|=0.\]
Thus $(f(\alpha_{k}))$ is a $\rho$-statistically-downward quasi-Cauchy sequence. This completes the proof of the theorem.
\end{proof}

Now we have the following result related to uniform convergence, namely, the uniform limit of a sequence of $\rho$-statistically-downward continuous functions is  $\rho$-statistically-downward continuous.
\begin{Thm} \label{TheoremUniformlimitofdownwardcont} If $(f_{n})$ is a sequence of $\rho$-statistically-downward continuous functions defined on a subset $E$ of $\mathbb{R}$ and $(f_{n})$ is uniformly convergent to a function $f$, then $f$ is $\rho$-statistically-downward continuous on $E$.
\end{Thm}
\begin{proof}
Let $\varepsilon$ be a positive real number and $(\alpha_{k})$ be any $\rho$-statistically-downward quasi-Cauchy sequence of points in $E$. By the uniform convergence of $(f_{n})$ there exists a positive integer $N$ such that $|f_{n}(x)-f(x)|<\frac{\varepsilon}{3}$ for all $x \in {E}$ whenever $n\geq N$. As $f_{N}$ is $\rho$-statistically-downward continuous on $E$, we have $$\lim_{n\rightarrow\infty} \frac{1}{\rho_{n}} |\{k\leq n: f_{N}(\alpha_{k+1})-f_{N}(\alpha_{k}) \geq \frac{\varepsilon}{3}\}|=0.$$ On the other hand, we have \\ $\{k\leq n: f(\alpha_{k+1})-f(\alpha_{k}) \geq \varepsilon\} \subseteq {\{k\leq n:  f(\alpha_{k+1})-f_{N}(\alpha_{k+1})\geq \frac{\varepsilon}{3}\}}$  \; \; \; \; \; \; \; \; \; \; \; $\cup \{k\leq n: f_{N}(\alpha_{k+1})-f_{N}(\alpha_{k}) \geq \frac{\varepsilon}{3}\} \cup \{k\leq n: f_{N}(\alpha_{k})-f(\alpha_{k}) \geq \frac{\varepsilon}{3}\}$
\\Hence \\ $\frac{1}{\rho_{n}} |\{k\leq n: f(\alpha_{k+1})-f(\alpha_{k}) \geq \varepsilon\}| \leq \frac{1}{\rho_{n}} |{k\leq n:  f(\alpha_{k+1})-f_{N}(\alpha_{k+1})\geq \frac{\varepsilon}{3}\}}|$  \; \; \; \; \; \; \; \; \; \; \; $+ \frac{1}{\rho_{n}} |\{k\leq n: f_{N}(\alpha_{k+1})-f_{N}(\alpha_{k}) \geq \frac{\varepsilon}{3}\}| + \frac{1}{\rho_{n}} |\{k\leq n:  f_{N}(\alpha_{k})-f(\alpha_{k}) \geq \frac{\varepsilon}{3}\}|$\\
Now it follows from this inequality that
$$lim_{n\rightarrow\infty} \frac{1}{\rho_{n}} |\{k\leq n: f(\alpha_{k+1})-f(\alpha_{k}) \geq \varepsilon\}|=0.$$
This completes the proof of the theorem.
\end{proof}

\section{Conclusion}

The results in this paper include not only the related theorems on statistical downward continuity studied in \cite{CakalliUpwardanddownwardstatisticalcontinuities} as special cases, i.e. $\rho_{n}=n$ for each $n\in{\mathbb{N}}$, but also include results which are also new for statistical downward continuity. In this paper, mainly, a new type of continuity, namely the concept of ${\rho}$-statistical downward continuity of a real function, and ${\rho}$-statistical downward compactness of a subset of $\mathbb{R}$ are introduced and investigated. In this investigation we have obtained theorems related to ${\rho}$-statistical downward continuity, and uniform continuity. We also introduced and studied some other continuities involving ${\rho}$-statistical downward  quasi-Cauchy sequences, and convergent sequences of points in $\mathbb{R}$. It turns out that the set of ${\rho}$-statistical downward continuous functions on a above bounded subset of $\mathbb{R}$ is contained in the set of uniformly continuous functions. We suggest to investigate ${\rho}$-statistical downward continuity  of fuzzy functions or soft functions (see \cite{CakalliandPratul}, and \cite{KocinacSelectionpropertiesinfuzzymetricspaces} for the definitions and related concepts in fuzzy setting, and see \cite{ArasandSonmezandCakalliOnSoftMappings}, \cite{ErdemandArasandCakalliandSonmezSoftmatricesonsoftmultisetsinanoptimaldecisionprocess} related concepts in soft setting). We also suggest to investigate ${\rho}$-statistical downward  continuity via double sequences (see for example \cite{KocinacDoubleSequencesandSelections},  \cite{CakalliandPattersonFunctionspreservingslowlyoscillatingdoublesequences}, \cite{CakalliandSavasStatisticalconvergenceofdoublesequences}, \cite{PattersonandCakalliQuasiCauchydoublesequences} and \cite{PattersonandSavasAsymptoticEquivalenceofDoubleSequences} for the definitions and related concepts in the double sequences case). For another further study, we suggest to investigate ${\rho}$-statistical downward continuity in an asymmetric cone metric space since in a cone metric space the notion of an ${\rho}$-statistical downward quasi Cauchy sequence coincides with the notion of an $\rho$-statistically quasi Cauchy sequence, and therefore ${\rho}$-statistical downward continuity coincides with $\rho$-statistically-ward continuity (see \cite{CakalliandSonmezandGenc}, \cite{CakalliandSonmezSlowlyOscillatingContinuityinAbstractMetricSpaces}, \cite{AkdumanandOzelandKilicmanSomeremarksoniwardcontinuity},  \cite{PalandSavasandCakalliIconvergenceonconemetricspaces}, \cite{SonmezandCakalliConenormedspacesandweightedmeans}).

\section{Acknowledgment}
We acknowledge that some of the results in this paper were presented at the International Conference on Recent Advances in Pure and Applied Mathematics (ICRAPAM 2017) May 11-15, 2017, Palm Wings Ephesus  Resort  Hotel, Kusadasi - Aydin, TURKEY.

\end{document}